\documentclass{amsart}
\usepackage{amsmath}
\usepackage{amssymb}
\usepackage{amsfonts}

\theoremstyle{plain}
\newtheorem{theorem}{Theorem}[section]
\newtheorem{lemma}[theorem]{Lemma}
\newtheorem{proposition}[theorem]{Proposition}

\newtheorem{corollary}[theorem]{Corollary}
\theoremstyle{definition}
\newtheorem{definition}[theorem]{Definition}
\newtheorem{example}[theorem]{Example}
\theoremstyle{remark}

\begin{document}
\title{CHARACTERIZATIONS OF HEMIRINGS BY THEIR $h$-IDEALS}
\author{Wieslaw A. Dudek}
\address{Institute of Mathematics and Computer Science, Wroclaw University
of Technology, 50-370 Wroclaw, Poland}
\email{dudek@im.pwr.wroc.pl}
\author{Muhammad Shabir}
\address{Department of Mathematics, Quaid-i-Azam University Islamabad.
Pakistan}
\email{mshabirbhatti@yahoo.co.uk}
\author{Rukhshanda Anjum}
\address{Department of Mathematics, Quaid-i-Azam University Islamabad.
Pakistan}
\email{rukh\_28@yahoo.com }

\begin{abstract}
In this paper we characterize hemirings in which all $h$-ideals or
all fuzzy $h$-ideals are idempotent. It is proved, among other
results, that every $h$-ideal of a hemiring $R$ is idempotent if
and only if the lattice of fuzzy $h$-ideals of $R$ is distributive
under the sum and $h$-intrinsic product of fuzzy $h$-ideals or, equivalently,
if and only if each fuzzy $h$-ideal of $R$ is intersection of those prime fuzzy
$h$-ideals of $R$ which contain it. We also define two types of
prime fuzzy $h$-ideals of $R$ and prove that, a non-constant
$h$-ideal of $R$ is prime in the second sense if and only if each of
its proper level set is a prime $h$-ideal of $R$.
\end{abstract}

 \maketitle


\section{Introduction}

The notion of semiring was introduced by H. S. Vandiver in 1934 \cite{22}.
Semirings which provide a common generalization of rings and distributive
lattices appear in a natural manner in some applications to the theory of
automata, formal languages, optimization theory and other branches of
applied mathematics (see for example \cite{3,7,9,13,18}). Hemirings, as
semirings with commutative addition and zero element, have also proved to be
an important algebraic tool in theoretical computer science (see for
instance \cite{5,12}). Some other applications of semirings with references
can be found in \cite{11,12,13}. On the other hand, the notions of automata
and formal languages have been generalized and extensively studied in a
fuzzy frame work (cf. \cite{20,21,23}).

Ideals play an important role in the structure theory of hemirings and are
useful for many purposes. But they do not coincide with usual ring ideals.
For this reason many results in ring theory have no analogues in semirings
using only ideals. Henriksen defined in \cite{14} a more restricted class of
ideals in semirings, which is called the class of $k$-ideals. A more
restricted class of ideals has been given by Iizuka \cite{15}. However, in
an additively commutative semiring $R$, ideals of a semiring coincide with
ideals of a ring, provided that a semiring is a hemiring. Now we call this
ideal an $h$-ideal of a hemiring.

The notion of fuzzy sets was introduced by Zadeh \cite{25}. Later
it was applied to many branches of mathematics. Investigations of
fuzzy semirings were initiated in \cite{2} and \cite{1}. Fuzzy
$k$-ideals are studied in \cite{10,17,4}. Fuzzy $h$-ideals of a
hemiring are studied by many authors, for example
\cite{16,26,27,8,19,Ma1,Ma2}. In this paper we characterize
hemirings in which each $h$-ideal is idempotent. We also
characterize hemirings for which each fuzzy $h$-ideal is
idempotent.


\section{Preliminaries}

Recall that a \textit{semiring} is an algebraic system $(R,+,\cdot)$
consisting of a non-empty set $R$ together with two binary operations on $R$
called addition and multiplication (denoted in the usual manner) such that $%
(R,+)$ and $(R,\cdot )$ are semigroups and the following distributive laws:
\begin{equation*}
a\cdot \left( b+c\right) =a\cdot b+a\cdot c,\text{ and }\left( b+c\right)
\cdot a=b\cdot a+c\cdot a
\end{equation*}
are satisfied for all $a,b,c\in R.$

A semiring $\left( R,+,\cdot \right) $ is called a \textit{hemiring} if $%
(R,+)$ is a commutative semigroup with a \textit{zero}, i.e., with an
element $0\in R$ such that $a+0=0+a=a$ and $a\cdot 0=0\cdot a=0$ for all $%
a\in R$. By the \textit{identity} of a hemiring $(R,+,\cdot )$ we mean an
element $1\in R$ (if it exists) such that $1\cdot a=a\cdot 1=a$ for all $%
a\in R.$

A hemiring $\left( R,+,\cdot \right)$ with a commutative semigroup $%
(R,\cdot) $ is called \textit{commutative}.

A non-empty subset $I$ of a hemiring $R$ is called a \textit{left $($right$)$
ideal} of $R$ if $\left( i\right) $ $a+b\in I$ for all $a,b\in I$ and $%
\left( ii\right) $ $ra\in I\left( ar\in I\right) $ for all $a\in I$, $r\in
R. $ Obviously $0\in I$ for any left (right) ideal $I$ of $R.$

A non-empty subset $A$ of a hemiring $R$ is called an \textit{ideal} of $R$
if it is both a left and a right ideal of $R$. A left (right) ideal $A$ of a
hemiring $R$ is called a \textit{left $($right$)$ $k$-ideal} of $R$ if for
any $a,b\in A$ and $x\in R $ from $x+a=b$ it follows $x\in A.$ A left
(right) ideal $I$ of a hemiring $R$ is called a \textit{left $($right$)$ $h$%
-ideal} of $R$ if for any $a,b\in I$ and $x,y\in R$ from $x+a+y=b+y$ it
follows $x\in I.$ Every left (right) $h$-ideal is a left (respectively,
right) $k$-ideal. The converse is not true \cite{17}.

\begin{lemma}
\label{L2.1} The intersection of any collection of left $($right$)$ $h$%
-ideals in a hemiring $R$ also is a left $($right$)$ $h$-ideal of $R$.
\end{lemma}

By $h$-closure of a non-empty subset $A$ of a hemiring $R$ we mean the set
\begin{equation*}
\overline{A}=\left\{ x\in R\mid x+a+y=b+y \ \text{ for some} \ a,b\in A,\
y\in R\right\}.
\end{equation*}

It is clear that if $A$ is a left (right) ideal of $R$, then $\overline{A}$
is the smallest left (right) $h$-ideal of $R$ containing $A$. So, $\overline{%
A}=A$ for all left (right) $h$-ideals of $R$. Obviously $\overline{\overline{%
A}}=\overline{A}$ for each non-empty $A\subseteq R.$ Also $\overline{A}%
\subseteq \overline{B}$ for all $A\subseteq B\subseteq R.$

\begin{lemma}
\cite{27}\label{L2.2} $\overline{AB}=\overline{\overline{A}\;\overline{B}}$
for any subsets $A$, $B$ of a hemiring $R$.
\end{lemma}

\begin{lemma}
\cite{27}\label{L2.3} If $A$ and $B$ are, respectively, right and left $h$%
-ideals of a hemiring $R$, then
\begin{equation*}
\overline{AB}\subseteq A\cap B.
\end{equation*}
\end{lemma}

\begin{definition}
\cite{27} A hemiring $R$ is said to be \textit{$h$-hemiregular} if for each $%
a\in R$, there exist $x,y,z\in R$ such that $a+axa+z=aya+z$.
\end{definition}

\begin{lemma}
\cite{27}\label{L2.5} A hemiring $R$ is $h$-hemiregular if and only if for
any right $h$-ideal $A$ and any left $h$-ideal $B$, we have
\begin{equation*}
\overline{AB}=A\cap B.
\end{equation*}
\end{lemma}

Let $X$ be a non-empty set. By a \textit{fuzzy subset} $\mu $ of $X$ we mean
a membership function $\mu :X\rightarrow \lbrack 0,1]$. Im$\mu $ denotes the
set of all values of $\mu $. A fuzzy subset $\mu :X\rightarrow \left[ 0,1%
\right] $ is non-empty if there exist at least one $x\in X$ such that $\mu
(x)>0$. For any fuzzy subsets $\lambda $ and $\mu $ of $X$ we define
\begin{eqnarray*}
&&\lambda \leq \mu \ \Longleftrightarrow \lambda \left( x\right) \leq \mu
\left( x\right) , \\
&&(\lambda \wedge \mu )(x)=\lambda (x)\wedge \mu (x)=\min \{\lambda (x),\mu
(x)\}, \\
&&\left( \lambda \vee \mu \right) \left( x\right) =\lambda \left( x\right)
\vee \mu \left( x\right) =\max \{\lambda (x),\mu (x)\}
\end{eqnarray*}%
for all $x\in X.$

More generally, if $\{\lambda _{i}:i\in I\}$ is a collection of fuzzy
subsets of $X,$ then by the \textit{intersection} and the \textit{union} of
this collection we mean fuzzy subsets
\begin{eqnarray*}
&&\Big( \bigwedge _{i\in I}\lambda_{i}\Big)(x) =\bigwedge_{i\in
I}\lambda_{i}(x)=\underset{i\in I}{\inf}\,\{\lambda_i(x)\}, \\
&&\Big(\bigvee_{i\in I}\lambda_{i}\Big)(x) =\bigvee_{i\in I}\lambda_{i}(x)=%
\underset{i\in I}{\sup}\,\{\lambda_i(x)\},
\end{eqnarray*}
respectively.

A fuzzy subset $\lambda $ of a semiring $R$ is called a \textit{fuzzy left $%
( $right$)$ ideal} of $R$ if for all $a,b\in R$ we have

\begin{enumerate}
\item[$(1)$] \ $\lambda \left( a+b\right)\geq\lambda(a)\wedge\lambda(b)$,

\item[$(2)$] \ $\lambda \left( ab\right) \geq \lambda(b), \
(\lambda(ab)\geq\lambda(a)).$
\end{enumerate}

Note that $\lambda(0)\geq\lambda(x)$ for all $x\in R.$

\begin{definition}
A fuzzy left (right) ideal $\lambda $ of a hemiring $R$ is called a \textit{%
fuzzy left $($right$)$ }

\textit{$\bullet$ \ $k$-ideal} if $x+y=z\longrightarrow \lambda \left(
x\right) \geq \lambda(y)\wedge\lambda(z)$,

$\bullet$ \ \textit{$h$-ideal} if $x+a+y=b+y\longrightarrow\lambda(x)\geq
\lambda(a)\wedge\lambda(b)$

\noindent holds for all $a,b,x,y\in R$.
\end{definition}

Properties of fuzzy sets defined on an algebraic system $\mathfrak{A}=(X,%
\mathbb{F})$, where $\mathbb{F}$ is a family of operations (also partial)
defined on $X$, can be characterized by the corresponding properties of some
subsets of $X$. Namely, as it is proved in \cite{KD} the following Transfer
Principle holds.

\begin{lemma}
\label{trans} A fuzzy set $\lambda$ defined on $\mathfrak{A}$ has the
property $\mathcal{P}$ if and only if all non-empty subsets $%
U(\lambda;t)=\{x\in X\,|\,\lambda(x)\ge t\}$ have the property $\mathcal{P}$.
\end{lemma}

For example, a fuzzy set $\lambda$ of a hemiring $R$ is a fuzzy left ideal
if and only if each non-empty subset $U(\lambda;t)$ is a left ideal of $R$.
Similarly, a fuzzy set $\lambda$ in a hemiring $R$ is a fuzzy left $h$-ideal
of $R$ if and only if each non-empty subset $U(\lambda;t)$ is a left $h$%
-ideal of $R$.

As a simple consequence of the above property, we obtain the following
proposition, which was first proved in \cite{16}.

\begin{proposition}
\label{P2.8} Let $A$ be a non-empty subset of a hemiring $R$. Then a fuzzy
set $\lambda_A$ defined by
\begin{equation*}
\lambda_A(x)=\left\{%
\begin{array}{ll}
t & \mathit{if } \ x\in A \\[2pt]
s & \mathit{otherwise}%
\end{array}%
\right.
\end{equation*}
where $0\le s<t\le 1$, is a fuzzy left $h$-ideal of $R$ if and only if $A$
is a left $h$-ideal of $R$.
\end{proposition}

\begin{proposition}
\label{P2.9} If \ \textrm{Im}$\lambda_A=\mathrm{Im}\lambda_B$ then

\medskip $(1)$ \ $A\subseteq B\longleftrightarrow\lambda_{A}\leq\lambda_{B}$,

\medskip $(2)$ \ $\lambda_{A}\wedge\lambda_{B}=\lambda_{A\cap B}$.
\end{proposition}

\begin{proof}
Let $A\subseteq B$. For $x\in A$ we have $\lambda_A(x)=t=\lambda_B(x)$. If $%
x\not\in A$, then $\lambda_A(x)=s\leq\lambda_B(x)$. So, $\lambda_A\leq%
\lambda_B$. Conversely, if $\lambda_A\leq\lambda_B$, then for all $x\in A$
we obtain $t=\lambda_A(x)\leq\lambda_B(x)$. Thus $\lambda_B(x)=t$, i.e., $%
x\in B$. Consequently, $A\subseteq B$. This proves $(1)$.

To prove $(2)$ let $x\in A\cap B$. Then $x\in A$, $x\in B$ and $%
\lambda_A(x)\wedge\lambda_B(x)=t=\lambda_{A\cap B}$. If $x\not\in A\cap B$,
then $\lambda_A(x)=s$ or $\lambda_B(x)=s$. So, $\lambda_A(x)\wedge%
\lambda_B(x)=s=\lambda_{A\cap B}(x)$, which completes the proof.
\end{proof}

\begin{definition}
\cite{16} Let $\lambda $ and $\mu $ be fuzzy subsets of a hemiring $R$. Then
the $h$-product of $\lambda $ and $\mu $ is defined by
\end{definition}

\begin{center}
\begin{tabular}{l}
$\left( \lambda \circ _{h}\mu \right) \left( x\right) =\left\{
\begin{tabular}{l}
$\underset{_{x+a_{1}b_{1}+y=a_{2}b_{2}+y}}{\sup}\!\!\!\!\!\!\!\!\!\big(
\lambda\left(a_{1}\right)\wedge\lambda\left(a_{2}\right) \wedge\mu\left(
b_{1}\right)\wedge\mu\left(b_{2}\right)\big)$ \\
$0 \ \ \ \ \text{ \ if }x\text{ is not expressed as }%
x+a_{1}b_{1}+y=a_{2}b_{2}+y$.%
\end{tabular}
\right. $%
\end{tabular}
\end{center}

\medskip

One can prove that if $\lambda$ and $\mu$ are fuzzy left (right) $h$-ideals
in a hemiring $R$, then so is $\lambda\wedge\mu$. Moreover, if $\lambda $ is
a fuzzy right $h$-ideal and $\mu $ is a fuzzy left $h$-ideal of $R.$ then $%
\lambda\circ_{h}\mu\leq\lambda \wedge\mu.$

\begin{theorem}
\label{T2.11}\cite{27} A hemiring $R$ is $h$-hemiregular if and only if \ $%
\lambda\circ_{h}\mu=\lambda\wedge\mu$ for any fuzzy right $h$-ideal $\lambda$
and fuzzy left $h$-ideal $\mu$.
\end{theorem}


\section{$h$-intrinsic product of fuzzy subsets}

To avoid repetitions from now $R$ will always mean a hemiring $(R,+,\cdot)$.

Generalizing the concept of $h$-product of two fuzzy subsets of $R$, in \cite%
{24} the following $h$-intrinsic product of fuzzy subsets is defined:

\begin{definition}
The \textit{$h$-intrinsic product} of two fuzzy subsets $\mu$ and $\nu$ on $%
R $ is defined by
\begin{equation*}
(\mu\odot_{h}\nu)(x)= \underset{x+\sum\limits_{i=1}^{m}\!\!a_{i}b_{i}+z=
\sum\limits_{j=1}^{n}\!\!a_{j}^{^{\prime }}b_{j}^{^{\prime }}+z}{\sup}\Big( %
\bigwedge\limits_{i=1}^m\big(\mu(a_{i})\wedge\nu(b_{i})\big)\wedge
\bigwedge\limits_{j=1}^n\big(\mu(a_{j}^{^{\prime
}})\wedge\nu(b_{j}^{^{\prime }}\big)\Big)
\end{equation*}
and $(\mu \odot _{h}\nu )(x)=0$ \ if $\,x\,$ cannot be expressed as \ $%
x+\sum\limits_{i=1}^{m}\!\!a_{i}b_{i}+z=\sum\limits_{j=1}^{n}\!\!a_{j}^{^{%
\prime }}b_{j}^{^{\prime }}+z$.
\end{definition}

The following properties of the $h$-intrinsic product of fuzzy sets proved
in \cite{24} will be used in this paper.

\begin{proposition}
\label{P3.2} Let $\mu$, $\nu$, $\omega$, $\lambda$ be fuzzy subsets on $R$.
Then

\medskip $(1)$ \ $\mu\circ_{h}\nu\leq\mu\odot_{h}\nu$,

\medskip $(2)$ \ $\mu\leq\omega$ and $\nu\leq\lambda \longrightarrow
\mu\odot_{h}\nu\leq\omega\odot_{h}\lambda$.

\medskip $(3)$ \ $\chi_{A}\odot_{h}\chi_{B}=\chi\overline{_{AB}}$ \ for
characteristic functions of any subsets of $R$.
\end{proposition}

\begin{theorem}
\label{T3.3} If $\lambda$ and $\mu$ are fuzzy $h$-ideals of $R$, then so is $%
\lambda\odot_{h}\mu$. Moreover, $\lambda\odot_{h}\mu\leq\lambda\wedge\mu$.
\end{theorem}
\begin{proof}
Let $\lambda $ and $\mu $ be fuzzy $h$-ideals of $R$. Let $x,y\in R$, then
\begin{equation*}
(\lambda\odot_{h}\mu)(x)= \underset{x+\sum\limits_{i=1}^{m}\!\!a_{i}b_{i}+z=
\sum\limits_{j=1}^{n}\!\!a_{j}^{^{\prime }}b_{j}^{^{\prime }}+z}{\sup}\Big( %
\bigwedge\limits_{i=1}^m\big(\lambda(a_{i})\wedge\mu(b_{i})\big)\wedge
\bigwedge\limits_{j=1}^n\big(\lambda(a_{j}^{^{\prime
}})\wedge\mu(b_{j}^{^{\prime }}\big)\Big)
\end{equation*}
and
\begin{equation*}
(\lambda\odot_{h}\mu)(y)= \underset{y+\sum\limits_{k=1}^{p}\!%
\!c_{k}d_{k}+z^{\prime }= \sum\limits_{l=1}^{q}\!\!c_{l}^{^{\prime
}}d_{l}^{^{\,\prime }}+z^{\prime }}{\sup}\Big( \bigwedge\limits_{k=1}^p\big(%
\lambda(c_{k})\wedge\mu(d_{k})\big)\wedge \bigwedge\limits_{l=1}^q\big(%
\lambda(c_{l}^{^{\prime }})\wedge\mu(d_{l}^{^{\,\prime }}\big)\Big)
\end{equation*}
Thus

\begin{equation*}
(\lambda\odot_{h}\mu)(x+y)= \underset{x+y+\sum\limits_{s=1}^{u}\!%
\!e_{s}f_{s}+z=\sum\limits_{t=1}^{v}\!\!e_{t}^{^{\prime}} f_{t}^{^{\prime
}}+z}{\sup}\Big(\bigwedge\limits_{s=1}^u\big(\lambda(e_{s})\wedge\mu(f_s)%
\big)\wedge \bigwedge\limits_{t=1}^v\big(\lambda (e_{t}^{^{\prime }})\wedge
\mu(f_{t}^{^{\prime}})\big)\Big)\geq
\end{equation*}
\begin{equation*}
\underset{x+\sum\limits_{i=1}^{m}\!\!a_{i}b_{i}+z=\sum\limits_{j=1}^{n}\!%
\!a_{j}^{^{ \prime }}b_{j}^{^{\prime }}+z}{\sup}\!\left( \underset{%
y+\sum\limits_{k=1}^{p}\!\!c_{k}d_{k}+z^{^{\prime
}}=\sum\limits_{l=1}^{q}\!\!c_{l}^{^{\prime }}d_{l}^{^{\,\prime
}}+z^{^{\prime}}}{\sup}\!\! \left(\!\!%
\begin{array}{l}
\!\!\bigwedge\limits_{i=1}^m\!\!\Big(\!\lambda(a_{i})\!\wedge\!\mu(b_{i})\!%
\Big)\!\!\wedge \!\bigwedge\limits_{j=1}^n\!\!\Big(\!\lambda
(a_{j}^{^{\prime }})\!\wedge\mu(b_{j}^{^{\prime}})\!\Big)\wedge \\
\bigwedge\limits_{k=1}^p\!\!\Big(\!\lambda(c_{k})\!\wedge\mu (d_{k})\!\Big)%
\!\!\wedge\!\!\bigwedge\limits_{l=1}^q\!\!\Big(\!\lambda(c_{l}^{^{\prime
}})\!\wedge\mu(d_{l}^{^{\,\prime }})\!\Big)%
\end{array}%
\!\!\!\!\right)\!\!\!\right)
\end{equation*}
\begin{equation*}
=\underset{x+\sum\limits_{i=1}^{m}\!\!a_{i}b_{i}+z=
\sum\limits_{j=1}^{n}\!\!a_{j}^{^{\prime }}b_{j}^{^{\prime }}+z}{\sup}\Big( %
\bigwedge\limits_{i=1}^m\big(\lambda(a_{i})\wedge\mu(b_{i})\big)\wedge
\bigwedge\limits_{j=1}^n\big(\lambda(a_{j}^{^{\prime
}})\wedge\mu(b_{j}^{^{\prime }}\big)\Big)\wedge\rule{60mm}{0mm}%
\end{equation*}
\begin{equation*}
\rule{30mm}{0mm}\underset{y+\sum\limits_{k=1}^{p}\!\!c_{k}d_{k}+z^{\prime }=
\sum\limits_{l=1}^{q}\!\!c_{l}^{^{\prime }}d_{l}^{^{\,\prime }}+z^{\prime }}{%
\sup}\Big( \bigwedge\limits_{k=1}^p\big(\lambda(c_{k})\wedge\mu(d_{k})\big)%
\wedge \bigwedge\limits_{l=1}^q\big(\lambda(c_{l}^{^{\prime
}})\wedge\mu(d_{l}^{^{\,\prime }}\big)\Big)
\end{equation*}
\begin{equation*}
=(\lambda \odot _{h}\mu )(x)\wedge (\lambda \odot _{h}\mu )(y).\rule%
{80mm}{0mm}%
\end{equation*}
Similarly,
\begin{equation*}
(\lambda \odot _{h}\mu )(xr)=\underset{xr+\sum\limits_{k=1}^{p}\!%
\!g_{k}h_{k}+z= \sum\limits_{l=1}^{q}\!\!g_{l}^{^{\prime }}h_{l}^{^{\prime
}}+z}{\sup}\Big( \bigwedge\limits_{k=1}^p\big(\lambda(g_{k})\wedge\mu(h_{k})%
\big)\wedge \bigwedge\limits_{l=1}^q\big(\lambda(g_{l}^{^{\prime
}})\wedge\mu(h_{l}^{^{\prime }}\big)\Big)\geq
\end{equation*}
\begin{equation*}
\underset{x+\sum\limits_{i=1}^{m}\!\!a_{i}b_{i}+z=
\sum\limits_{j=1}^{n}\!\!a_{j}^{^{\prime }}b_{j}^{^{\prime }}+z}{\sup}\Big( %
\bigwedge\limits_{i=1}^m\big(\lambda(a_{i})\wedge\mu(b_{i}r)\big)\wedge
\bigwedge\limits_{j=1}^n\big(\lambda(a_{j}^{^{\prime
}})\wedge\mu(b_{j}^{^{\prime }}r\big)\Big)\geq\rule{30mm}{0mm}%
\end{equation*}
\begin{equation*}
\underset{x+\sum\limits_{i=1}^{m}\!\!a_{i}b_{i}+z=
\sum\limits_{j=1}^{n}\!\!a_{j}^{^{\prime }}b_{j}^{^{\prime }}+z}{\sup}\Big( %
\bigwedge\limits_{i=1}^m\big(\lambda(a_{i})\wedge\mu(b_{i})\big)\wedge
\bigwedge\limits_{j=1}^n\big(\lambda(a_{j}^{^{\prime
}})\wedge\mu(b_{j}^{^{\prime }}\big)\Big)=(\lambda\odot_{h}\mu)(x).\rule%
{30mm}{0mm}%
\end{equation*}
Analogously we can verify that $(\lambda\odot_{h}\mu)(rx)\geq
(\lambda\odot_{h}\mu)(x)$ for all $r\in R$. This means that $%
\lambda\odot_{h}\mu$ is a fuzzy ideal of $R.$

To prove that $x+a+y=b+y$ implies $(\lambda\odot_{h}\mu)(x)\geq(\lambda%
\odot_{h}\mu)(a)\wedge (\lambda\odot_{h}\mu)(b)$ observe that
\begin{equation}  \label{e1}
a+\sum\limits_{i=1}^{m}a_{i}b_{i}+z_{1}=\sum\limits_{j=1}^{n}a_{j}^{^{\prime
}}b_{j}^{^{\prime }}+z_{1} \ \ \mathrm{and } \ \
b+\sum\limits_{k=1}^{l}c_{k}d_{k}+z_{2}=\sum\limits_{q=1}^{p}c_{q}^{^{\prime
}}d_{q}^{^{\,\prime }}+z_{2},
\end{equation}
together with $x+a+y=b+y$, gives $x+a+(\sum%
\limits_{i=1}^{m}a_{i}b_{i}+z_{1})+y=
b+(\sum\limits_{i=1}^{m}a_{i}b_{i}+z_{1})+y$. Thus $x+\sum%
\limits_{j=1}^{n}a_{j}^{^{\prime }}b_{j}^{^{\prime
}}+z_{1}+y=b+\sum\limits_{i=1}^{m}a_{i}b_{i}+z_{1}+y$ and, consequently, $%
x+\sum\limits_{j=1}^{n}a_{j}^{^{\prime }}b_{j}^{^{\prime
}}+(\sum\limits_{k=1}^{l}c_{k}d_{k}+z_{2})+z_{1}+y=
b+(\sum\limits_{k=1}^{l}c_{k}d_{k}+z_{2})+
\sum\limits_{i=1}^{m}a_{i}b_{i}+z_{1}+y=\sum\limits_{q=1}^{p}c_{q}^{^{\prime
}}d_{q}^{^{\,\prime}}+z_2+\sum\limits_{i=1}^{m}a_{i}b_{i}+z_1+y=
\sum\limits_{i=1}^{m}a_{i}b_{i}+\sum\limits_{q=1}^{p}c_{q}^{^{\prime
}}d_{q}^{^{\,\prime}}+z_2+z_1+y$.\ Therefore
\begin{equation}  \label{e2}
x+\sum\limits_{j=1}^{n}a_{j}^{^{\prime }}b_{j}^{^{\prime
}}+\sum\limits_{k=1}^{l}c_{k}d_{k}+z_{2}+z_{1}+y=
\sum\limits_{i=1}^{m}a_{i}b_{i}+\sum\limits_{q=1}^{p}c_{q}^{^{\prime
}}d_{q}^{^{\,\prime}}+z_2+z_1+y .
\end{equation}

Now, in view of \eqref{e1} and \eqref{e2}, we have
\begin{equation*}
\begin{array}{l}
(\lambda \odot_{h}\mu )(a)\wedge (\lambda \odot _{h}\mu )(b)= \\[3pt]
\underset{a+\sum\limits_{i=1}^{m}\!\!a_{i}b_{i}+z=\sum\limits_{j=1}^{n}\!%
\!a_{j}^{^{\prime }}b_{j}^{^{\prime }}+z}{\sup }\Big(\bigwedge%
\limits_{i=1}^{m}\big(\lambda (a_{i})\wedge \mu (b_{i})\big)\wedge
\bigwedge\limits_{j=1}^{n}\big(\lambda (a_{j}^{^{\prime }})\wedge \mu
(b_{j}^{^{\prime }}\big)\Big)\wedge \\
\rule{20mm}{0mm}\underset{b+\sum\limits_{k=1}^{p}\!\!c_{k}d_{k}+z^{\prime
}=\sum\limits_{l=1}^{q}\!\!c_{l}^{^{\prime }}d_{l}^{^{\,\prime }}+z^{\prime }%
}{\sup }\Big(\bigwedge\limits_{k=1}^{p}\big(\lambda (c_{k})\wedge \mu (d_{k})%
\big)\wedge \bigwedge\limits_{l=1}^{q}\big(\lambda (c_{l}^{^{\prime
}})\wedge \mu (d_{l}^{^{\,\prime }}\big)\Big) \\
=\underset{a+\sum\limits_{i=1}^{m}\!\!a_{i}b_{i}+z=\sum\limits_{j=1}^{n}\!%
\!a_{j}^{^{\prime }}b_{j}^{^{\prime }}+z}{\sup }\!\!\left( \underset{%
b+\sum\limits_{k=1}^{p}\!\!c_{k}d_{k}+z^{^{\prime
}}=\sum\limits_{l=1}^{q}\!\!c_{l}^{^{\prime }}d_{l}^{^{\,\prime
}}+z^{^{\prime }}}{\sup }\!\!\left( \!\!%
\begin{array}{l}
\!\!\bigwedge\limits_{i=1}^{m}\!\!\Big(\!\lambda (a_{i})\!\wedge \!\mu
(b_{i})\!\Big)\!\!\wedge \!\bigwedge\limits_{j=1}^{n}\!\!\Big(\!\lambda
(a_{j}^{^{\prime }})\!\wedge \mu (b_{j}^{^{\prime }})\!\Big)\wedge \\
\bigwedge\limits_{k=1}^{p}\!\!\Big(\!\lambda (c_{k})\!\wedge \mu (d_{k})\!%
\Big)\!\!\wedge \!\!\bigwedge\limits_{l=1}^{q}\!\!\Big(\!\lambda
(c_{l}^{^{\prime }})\!\wedge \mu (d_{l}^{^{\,\prime }})\!\Big)%
\end{array}%
\!\!\!\!\right) \!\!\!\right) \\[3pt]
\leq \underset{x+\sum\limits_{s=1}^{u}g_{s}h_{s}+z=\sum%
\limits_{t=1}^{w}g_{t}^{^{\prime }}h_{t}^{^{\prime }}+z}{\sup }\!\!\left(
\bigwedge\limits_{s=1}^{u}\!\!\Big(\lambda (g_{s})\wedge \mu (h_{s})\Big)%
\wedge \bigwedge\limits_{t=1}^{w}\Big(\lambda (g_{t}^{^{\prime }})\!\wedge
\mu (h_{t}^{^{\prime }})\!\Big)\right) =(\lambda \odot _{h}\mu )(x).%
\end{array}%
\end{equation*}

Thus $(\lambda \odot _{h}\mu )(a)\wedge (\lambda \odot _{h}\mu )(b)\leq
(\lambda \odot _{h}\mu )(x)$. This completes the proof that $(\lambda \odot
_{h}\mu )$ is a fuzzy $h$-ideal of $R$.

By simple calculations we can prove that $\lambda \odot _{h}\mu \leq \lambda
\wedge \mu $.
\end{proof}

For $h$-hemiregular hemirings we have stronger result. Namely, as it is
proved in \cite{24}, the following theorem is valid.

\begin{theorem}
\label{T3.4} A hemiring $R$ is $h$-hemiregular if and only if for any fuzzy
right $h$-ideal $\lambda $ and any fuzzy left $h$-ideal $\mu $ of $R$ we have
$\lambda \odot _{h}\mu =\lambda\wedge\mu $.
\end{theorem}

Comparing this theorem with Theorem \ref{T2.11} we obtain

\begin{corollary}
$\lambda\odot_{h}\mu =\lambda\circ\mu$ for all fuzzy $h$-ideals of any $h$%
-hemiregular hemiring.
\end{corollary}


\section{Idempotent $h$-ideals}

The concept of h-hemiregularity of a hemiring was introduced in \cite{27} as
a generalization of the concept of regularity of a ring. From results proved
in \cite{27} (see our Lemma \ref{L2.5}) it follows that in $h$-hemiregular hemirings every $h$-ideal $A$ is \textit{$h$-idempotent}, that
is $\overline{AA}=A$. On the other hand, Theorem \ref{T3.4} implies that in
such hemirings we have $\lambda\odot_{h}\lambda=\lambda$ for all fuzzy $h$%
-ideals. Fuzzy $h$-ideals with this property will be called \textit{%
idempotent}.

\begin{proposition}
\label{P4.1} The following statements are equivalent:

\begin{enumerate}
\item Each $h$-ideal of $R$ is $h$-idempotent.

\item $A\cap B=\overline{AB}$ for each pair of $h$-ideals of $R.$

\item $x\in\overline{RxRxR}$ for every $x\in R$.

\item $A\subseteq\overline{RARAR}$ for every non-empty $A\subseteq R$.

\item $A=\overline{RARAR}$ for every $h$-ideal $A$ of $R$.
\end{enumerate}
\end{proposition}

\begin{proof}
Indeed, by Lemma \ref{L2.3}, $\overline{AB}\subseteq A\cap B$ for all $h$%
-ideals of $R$. Since $A\cap B$ is an $h$-ideal of $R,$ $(1)$ implies $A\cap
B=\overline{(A\cap B)(A\cap B)}\subseteq\overline{AB}$. Thus $A\cap B=%
\overline{AB}$. So, $(1)$ implies $(2)$. The converse implication is obvious.

It is clear that the smallest $h$-ideal of $R$ containing $x\in R$ has the
form
\begin{equation*}
\langle x\rangle=\overline{\langle x\rangle}=\overline{Rx+xR+RxR+Sx},
\end{equation*}
where $Sx$ is a finite sum of $x$'s. If $(1)$ holds, then $\overline{\langle
x\rangle}= \overline{\overline{\langle x\rangle}\;\overline{\langle x\rangle}%
}= \overline{\langle x\rangle\langle x\rangle}$. Consequently,
\begin{equation*}
\arraycolsep.5mm%
\begin{array}{ll}
x & =0+x\in\overline{Rx+xR+RxR+Sx} \\
& =\overline{(Rx+xR+RxR+Sx)(Rx+xR+RxR+Sx)} \subseteq\overline{RxRRxR}%
\subseteq\overline{RxRxR}%
\end{array}%
\end{equation*}
for every $x\in R$. So, $(1)$ implies $(3)$. Clearly $(3)$ implies $(4)$. If
$(4)$ holds, then for every $h$-ideal of $R$ we have $\overline{A}=A\subseteq%
\overline{RARAR}\subseteq\overline{AA}\subseteq\overline{A}=A$, which proves
$(5)$. The implication $(5)\to (1)$ is obvious.
\end{proof}

As a consequence of the above result and Lemma \ref{L2.5} we obtain the
following characterization of $h$-hemiregularity of commutative hemirings.

\begin{corollary}
A commutative hemiring is $h$-hemiregular if and only if all its $h$-ideals
are $h$-idempotent.
\end{corollary}

\begin{proposition}
\label{P4.3} The following statements are equivalent:

\begin{enumerate}
\item Each fuzzy $h$-ideal of $R$ is idempotent.

\item $\lambda\odot_{h}\mu=\lambda\wedge\mu$ for all fuzzy $h$-ideals of $R$.
\end{enumerate}
\end{proposition}

\begin{proof}
Let $\lambda$ and $\mu$ be fuzzy $h$-ideals of $R$. Since $\lambda\wedge\mu$
is a fuzzy $h$-ideal of $R$ such that $\lambda\wedge\mu\leq\lambda$ and $%
\lambda\wedge\mu\leq\mu$, Proposition \ref{P3.2} implies $%
(\lambda\wedge\mu)\odot_{h}(\lambda\wedge\mu)\leq\lambda\odot_{h}\mu$. So,
if $\lambda\wedge\mu$ is an idempotent fuzzy $h$-ideal, then $%
\lambda\wedge\mu\leq\lambda\odot_{h}\mu$, which together with Theorem \ref%
{T3.3} gives $\lambda\odot_{h}\mu=\lambda\wedge\mu$. This means that $(1)$
implies $(2)$. The converse implication is obvious.
\end{proof}

Comparing this proposition with Theorem \ref{T3.4} we obtain

\begin{corollary}
A commutative hemiring is $h$-hemiregular if and only if all its fuzzy $h$%
-ideals are idempotent, or equivalently, if and only if $\lambda\odot_{h}%
\mu=\lambda\wedge\mu$ holds for all its fuzzy $h$-ideals.
\end{corollary}

\begin{theorem}
For hemirings with the identity the following statements are equivalent:

\begin{enumerate}
\item Each $h$-ideal of $R$ is $h$-idempotent.

\item $A\cap B=\overline{AB}$ for each pair of $h$-ideals of $R.$

\item Each fuzzy $h$-ideal of $R$ is idempotent.

\item $\lambda\odot_{h}\mu=\lambda\wedge\mu$ for all fuzzy $h$-ideals of $R$.
\end{enumerate}
\end{theorem}

\begin{proof}
$(1)$ and $(2)$ are equivalent by Proposition \ref{P4.1}, $(3)$ and $(4)$ by
Proposition \ref{P4.3}. To prove that $(1)$ and $(3)$ are equivalent observe
that the smallest $h$-ideal containing $x\in R$ has the form $RxR$. Its
closure $\overline{RxR}$ also is an $h$-ideal. Since, by $(1)$, all $h$%
-ideals of $R$ are $h$-idempotent, we have $\overline{RxR}=\overline{(%
\overline{RxR})(\overline{RxR)}}=\overline{RxRRxR}$ (Lemma \ref{L2.2}). Thus
$x\in\overline{RxR}=\overline{RxRRxR}$ implies
\begin{equation*}
x+\sum_{i=1}^{m}r_{i}xs_{i}u_{i}xt_{i}+z=\sum_{j=1}^{n}r_{j}^{^{\prime
}}xs_{j}^{^{\prime }}u_{j}^{^{\prime }}t_{j}^{^{\prime }}+z.
\end{equation*}
But, by Theorem \ref{T3.3}, for every fuzzy $h$-ideal of $R$ we have $%
\lambda\odot_{h}\lambda\leq\lambda$. Hence $\lambda(x)=\lambda(x)\wedge%
\lambda(x)\leq\bigwedge\limits_{i=1}^m \Big(\lambda(r_{i}xs_{i})\wedge%
\lambda(u_{i}xt_{i})\Big)$. Also $\lambda(x)=\lambda(x)\wedge\lambda(x)\leq%
\bigwedge\limits_{j=1}^n \left(\lambda(r_{j}^{^{\prime }}xs_{j}^{^{\prime
}})\wedge \lambda(u_{j}^{^{\prime }}xt_{j}^{^{\prime }})\right)$. Therefore

$\lambda (x)\leq \bigwedge\limits_{i=1}^m \Big(\lambda(r_{i}xs_{i})\wedge%
\lambda(u_{i}xt_{i})\Big)\wedge \bigwedge\limits_{j=1}^n
\left(\lambda(r_{j}^{^{\prime }}xs_{j}^{^{\prime }})\wedge
\lambda(u_{j}^{^{\prime }}xt_{j}^{^{\prime
}})\right)=M(x,r_i,s_i,r_j^{\prime },s_j^{\prime }) $

\medskip

$\leq\underset{x+\sum\limits_{i=1}^{m}r_{i}xs_{i}u_{i}xt_{i}+z=
\sum\limits_{j=1}^{n}r_{j}^{^{\prime }}xs_{j}^{^{\prime }}u_{j}^{^{\prime
}}t_{j}^{^{\prime }}+z}{\sup} M(x,r_i,s_i,r_j^{\prime },s_j^{\prime
})=(\lambda\odot_{h}\lambda)(x)$.

Hence $\lambda\leq\lambda\odot_{h}\lambda$, which proves $%
\lambda\odot_{h}\lambda=\lambda$. So, $(1)$ implies $(3)$.

Conversely, according to Proposition \ref{P2.8}, the characteristic function
$\chi_A$ of any $h$-ideal $A$ of $R$ is a fuzzy $h$-ideal of $R$. If it is
idempotent, then $\chi_{A}=\chi_A\odot_{h}\chi_A=\chi_{\overline{AA}}$
(Proposition \ref{P3.2}). Thus $A=\overline{AA}$. $(3)$ implies $(1)$.
\end{proof}

\begin{definition}
The \textit{$h$-sum} \ $\lambda+_{h}\mu$ \ of fuzzy subsets $\lambda$ and $%
\mu$ of $R$ is defined by
\begin{equation*}
\left(\lambda+_{h}\mu\right)(x)=\underset{x+(a_{1}+b_{1})+z=(a_{2}+b_{2})+z}{%
\sup} \Big(\lambda(a_{1})\wedge\lambda(a_{2})\wedge\mu(b_{1})\wedge\mu(b_{2})%
\Big),
\end{equation*}
where $x,a_{1},b_{1},a_{2},b_{2},z\in R.$
\end{definition}

\begin{theorem}
\label{T4.7} The $h$-sum of fuzzy $h$-ideals of $R$ also is a fuzzy $h$%
-ideal of $R$.
\end{theorem}
\begin{proof}
Let $\lambda$, $\mu$ be fuzzy $h$-ideals of $R$. Then for $x,y\in R$ we
have

\begin{eqnarray*}
&& (\lambda+_{h}\mu )(x)\wedge(\lambda+_{h}\mu)(y)= \\
&&\underset{x+(a_{1}+b_{1})+z=(a_{2}+b_{2})+z}{\sup} \Big(%
\lambda(a_{1})\wedge\lambda(a_{2})\wedge\mu(b_{1})\wedge\mu(b_{2})\Big)\wedge
\\
&&\rule{28mm}{0mm}\underset{y+(a_{1}^{\prime }+b_{1}^{\prime })+z^{\prime
}=(a_{2}^{\prime }+b_{2}^{\prime })+z^{\prime }}{\sup} \Big(%
\lambda(a_{1}^{\prime })\wedge\lambda(a_{2}^{\prime
})\wedge\mu(b_{1}^{\prime })\wedge\mu(b_{2}^{\prime })\Big) \\
&&=\underset{{\scriptsize
\begin{array}{l}
x+( a_{1}+b_{1})+z=( a_{2}+b_{2}) +z \\
y+( a_{1}^{\prime }+b_{1}^{\prime })+z^{\prime }=(a_{2}^{\prime
}+b_{2}^{\prime })+z^{\prime }%
\end{array}%
} }{\sup}\left(
\begin{array}{c}
\lambda(a_{1})\wedge\lambda(a_{2})\wedge\mu(b_{1})\wedge\mu(b_{2})\wedge \\
\lambda(a_{1}^{\prime })\wedge\lambda(a_{2}^{\prime
})\wedge\mu(b_{1}^{\prime })\wedge\mu(b_{2}^{\prime })%
\end{array}
\right) \\
&&\leq\underset{{\scriptsize
\begin{array}{l}
x+( a_{1}+b_{1})+z=(a_{2}+b_{2})+z \\
y+( a_{1}^{\prime }+b_{1}^{\prime })+z^{\prime }=(a_{2}^{\prime
}+b_{2}^{\prime })+z^{\prime }%
\end{array}
}}{\sup}\left(
\begin{array}{c}
\lambda(a_{1}+a_1^{\prime })\wedge\lambda(a_2+a_2^{\prime })\wedge \\
\mu(b_{1}+b_1^{\prime })\wedge\mu(b_2+b_{2}^{\prime })%
\end{array}
\right) \\
&&\leq\underset{(x+y)+(c_{1}+d_{1})+z^{\prime \prime
}=(c_{2}+d_{2})+z^{\prime \prime }}{\sup} \Big(\lambda(c_{1})\wedge%
\lambda(c_{2})\wedge\mu(d_{1})\wedge\mu(d_{2})\Big) \\[8pt]
&&=\left(\lambda +_{h}\mu \right)(x+y).
\end{eqnarray*}

Similarly,
\begin{eqnarray*}
(\lambda+_{h}\mu)(x)&=&\underset{x+(a_{1}+b_{1})+z=(a_{2}+b_{2})+z}{\sup} %
\Big(\lambda(a_{1})\wedge\lambda(a_{2})\wedge\mu(b_{1})\wedge\mu(b_{2})\Big)
\\
&\leq &\underset{x+(a_{1}+b_{1})+z=(a_{2}+b_{2})+z}{\sup} \Big(%
\lambda(ra_{1})\wedge\lambda(ra_{2})\wedge\mu(rb_{1})\wedge\mu(rb_{2})\Big)
\\
&\leq &\underset{rx+(a_{1}^{\prime \prime }+b_{1}^{\prime \prime
})+z^{\prime \prime }=(a_{2}^{\prime \prime }+b_{2}^{\prime \prime
})+z^{\prime \prime }}{\sup} \Big(\lambda(a_{1}^{\prime \prime
})\wedge\lambda(a_{2}^{\prime \prime })\wedge\mu(b_{1}^{\prime \prime
})\wedge\mu(b_{2}^{\prime \prime })\Big) \\[8pt]
&=&(\lambda +_{h}\mu)(rx).
\end{eqnarray*}

Analogously $(\lambda+_{h}\mu)(x)\leq(\lambda+_{h}\mu)(xr)$. This proves
that $(\lambda+_{h}\mu)$ is a fuzzy ideal of $R$.

Now we show that $x+a+z=b+z$ implies $(\lambda +_{h}\mu)(x)\geq (\lambda
+_{h}\mu)(a)\wedge(\lambda +_{h}\mu)(b)$. For this let $%
a+(a_{1}+b_{1})+z_{1}=(a_{2}+b_{2})+z_{1}$ and $%
b+(c_{1}+d_{1})+z_{2}=(c_{2}+d_{2})+z_{2}$. Then,
\begin{equation*}
a+(c_{2}+d_{2}+z_{2})+(a_{1}+b_{1}+z_{1})=
(a_{2}+b_{2}+z_{1})+(b+c_{1}+d_{1}+z_{2}),
\end{equation*}
whence
\begin{equation*}
a+\left( a_{1}+c_{2}\right) +\left( b_{1}+d_{2}\right) +\left(
z_{1}+z_{2}\right) =b+\left( a_{2}+c_{1}\right) +\left( b_{2}+d_{1}\right)
+\left( z_{1}+z_{2}\right).
\end{equation*}
Consequently
\begin{equation*}
a+\left( a_{1}+c_{2}\right) +\left( b_{1}+d_{2}\right) +\left(
z_{1}+z_{2}+z\right) =b+z+\left( a_{2}+c_{1}\right) +\left(
b_{2}+d_{1}\right) +\left( z_{1}+z_{2}\right)
\end{equation*}
and
\begin{equation*}
a+\left( a_{1}+c_{2}\right) +\left( b_{1}+d_{2}\right) +\left(
z_{1}+z_{2}+z\right) =x+a+z+\left( a_{2}+c_{1}\right) +\left(
b_{2}+d_{1}\right) +\left( z_{1}+z_{2}\right).
\end{equation*}
Thus
\begin{equation*}
x+\left( a_{2}+c_{1}\right) +\left( b_{2}+d_{1}\right) +\left(
z_{1}+z_{2}+z+a\right) =\left( a_{1}+c_{2}\right) +\left( b_{1}+d_{2}\right)
+\left( z_{1}+z_{2}+z+a\right),
\end{equation*}
i.e., $x+(a^{\prime }+b^{\prime })+z^{\prime }=(a^{\prime \prime }+b^{\prime
\prime })+z^{\prime }$ for some $a^{\prime },b^{\prime },a^{\prime \prime
},b^{\prime \prime }\in R$.

\medskip

Therefore
\begin{eqnarray*}
&&\left( \lambda +_{h}\mu \right) \left( a\right) \wedge \left( \lambda
+_{h}\mu \right) \left( b\right)= \\
&&\underset{a+(a_{1}+b_{1})+z_{1}=(a_{2}+b_{2})+z_{1}}{\sup}\Big(\lambda
(a_{1})\wedge\lambda(a_{2})\wedge\mu(b_{1})\wedge\mu(b_{2})\Big)\wedge \\
&&\rule{27mm}{0mm}\underset{b+(c_{1}+d_{1})+z_{2}=(c_{2}+d_{2})+z_{2}}{\sup}%
\Big( \lambda(c_{1})\wedge\lambda(c_{2})\wedge\mu(d_{1})\wedge\mu(d_{2})\Big)
\\
&&=\underset{{\scriptsize
\begin{tabular}{l}
$a+(a_{1}+b_{1})+z_{1}=(a_{2}+b_{2})+z_{1}$ \\
$b+(c_{1}+d_{1})+z_{2}=(c_{2}+d_{2})+z_{2}$%
\end{tabular}
}}{\sup}\left(
\begin{array}{c}
\lambda(a_{1})\wedge\lambda(a_{2})\wedge\mu(b_{1})\wedge\mu(b_{2}) \wedge \\
\lambda(c_{1})\wedge\lambda(c_{2})\wedge\mu(d_{1})\wedge\mu(d_{2})%
\end{array}
\right) \\
&&\leq\underset{{\scriptsize
\begin{tabular}{l}
$a+(a_{1}+b_{1})+z_{1}=(a_{2}+b_{2})+z_{1}$ \\
$b+(c_{1}+d_{1})+z_{2}=(c_{2}+d_{2})+z_{2}$%
\end{tabular}
}}{\sup}\left(
\begin{array}{c}
\lambda \left( a_{1}+c_{2}\right) \wedge \lambda \left(
a_{2}+c_{1}\right)\wedge \\
\mu \left( b_{1}+d_{2}\right) \wedge \mu \left( b_{2}+d_{1}\right)%
\end{array}
\right) \\
&&\leq\underset{x+(a^{\prime }+b^{\prime })+z^{\prime }=(a^{\prime \prime
}+b^{\prime \prime })+z^{\prime }}{\sup} \Big(\lambda(a^{\prime
})\wedge\lambda(a^{\prime \prime })\wedge\mu(b^{\prime })\wedge\mu(b^{\prime
\prime })\Big) \\[8pt]
&&=(\lambda +_{h}\mu)(x) .
\end{eqnarray*}

Thus \ $\lambda +_{h}\mu$ \ is a fuzzy $h$-ideal of $R$.
\end{proof}

\begin{theorem}
\label{T4.8} If all $h$-ideals of $R$ are $h$-idempotent, then the
collection of these $h$-ideals forms a complete Brouwerian lattice.
\end{theorem}

\begin{proof}
The collection $\mathcal{L}_{R}$ of all $h$-ideals of $R$ is a poset under
the inclusion of sets. It is not difficult to see that $\mathcal{L}_{R}$ is
a complete lattice under operations $\sqcup$, $\sqcap$ defined as $A\sqcup B=%
\overline{A+B}$ and $A\sqcap B=A\cap B$.

We show that $\mathcal{L}_{R}$ is a Brouwerian lattice, that is, for any $%
A,B\in \mathcal{L}_{R},$ the set $\mathcal{L}_{R}(A,B)=\{I\in\mathcal{L}%
_{R}\,|\,A\cap I\subseteq B\}$ contains a greatest element.

By Zorn's Lemma the set $\mathcal{L}_{R}(A,B)$ contains a maximal element $M$%
. Since each $h$-ideal of $R$ is $h$-idempotent, $\overline{AI}=A\cap
I\subseteq B$ and $\overline{AM}=A\cap M\subseteq B$ (Proposition \ref{P4.1}%
). Thus $\overline{AI}+\overline{AM}\subseteq B$. Consequently, $\overline{%
\overline{AI}+\overline{AM}}\subseteq \overline{B}=B.$

Since $\overline{I+M}=I\sqcup M\in\mathcal{L}_{R}$, for every $x\in\overline{%
I+M}$ there exist $i_{1},i_{2}\in I,$ $m_{1},m_{2}\in M$ and $z\in R$ such
that $x+i_{1}+m_{1}+z=i_{2}+m_{2}+z$. Thus
\begin{equation*}
dx+di_{1}+dm_{1}+dz=di_{2}+dm_{2}+dz
\end{equation*}
for any $d\in D\in\mathcal{L}_{R}$. As $\,di_{1},di_{2}\in DI,$ $%
\,dm_{1},dm_{2}\in DM,$ $\,dz\in R,$ we have $dx\in \overline{DI+DM}$, which
implies $D\left(\overline{I+M}\right) \subseteq \overline{DI+DM}\subseteq
\overline{\overline{DI}+\overline{DM}}\subseteq B$. Hence $\overline{D\left(%
\overline{I+M}\right)}\subseteq B.$ This means that $D\cap\left(\overline{I+M%
}\right)=\overline{D\left( \overline{I+M}\right)}\subseteq B$, i.e., $%
\overline{I+M}\in\mathcal{L}_{R}(A,B)$, whence $\overline{I+M}=M$ because $M$
is maximal in $\mathcal{L}_{R}(A,B)$. Therefore $I\subseteq\overline{I}%
\subseteq \overline{I+M}=M$ for every $I\in \mathcal{L}_{R}(A,B)$.
\end{proof}

\begin{corollary}
\label{C4.9} If all $h$-ideals of $R$ are idempotent, then the lattice $%
\mathcal{L}_{R}$ is distributive.
\end{corollary}

\begin{proof}
Each complete Brouwerian lattice is distributive (cf. \cite{6}, 11.11).
\end{proof}

\begin{theorem}
\label{f-lattice} Each fuzzy $h$-ideal of $R$ is $h$-idempotent if and only
if the set of all fuzzy $h$-ideals of $R$ (ordered by $\leq )$ forms a
distributive lattice under the $h$-sum and $h$-intrinsic product of fuzzy $h$%
-ideals with \ $\lambda\odot_{h}\mu=\lambda\wedge\mu$.
\end{theorem}

\begin{proof}
Assume that all fuzzy $h$-ideals of $R$ are idempotent. Then $%
\lambda\odot_{h}\mu=\lambda\wedge\mu$ (Proposition \ref{P4.3}) and, as it is
not difficult to see, the set $\mathcal{FL}_{R}$ of all fuzzy $h$-ideals of $%
R$ (ordered by $\leq $) is a lattice under the $h$-sum and $h$-intrinsic
product of fuzzy $h$-ideals.

We show that \ $(\lambda\odot_{h}\delta)+_{h}\mu=(\lambda+_{h}\mu)\odot_{h}(%
\delta+_{h}\mu)$ \ for all $\lambda,\mu,\delta\in\mathcal{FL}_{R}$. Indeed,
for any $x\in R$ we have
\begin{eqnarray*}
&&\Big((\lambda\odot_{h}\delta)+_{h}\mu\Big)(x)=\Big((\lambda\odot_{h}%
\delta)+_{h}\mu\Big)(x) \\
&&=\underset{x+(a_{1}+b_{1})+z=(a_{2}+b_{2})+z}{\sup} \Big(%
(\lambda\wedge\delta)(a_{1})\wedge(\lambda\wedge\delta)(a_{2})
\wedge\mu(b_{1})\wedge\mu( b_{2})\Big) \\
&&=\underset{x+(a_{1}+b_{1})+z=(a_{2}+b_{2})+z}{\sup}\Big(\lambda
(a_{1})\wedge\lambda(a_{2})\wedge\mu(b_{1})\wedge\mu(b_{2})\wedge
\delta(a_{1})\wedge\delta(a_{2})\Big) \\
&&=\underset{x+(a_{1}+b_{1})+z=(a_{2}+b_{2})+z}{\sup} \Big(%
\lambda(a_{1})\wedge\lambda(a_{2})\wedge\mu(b_{1})\wedge\mu(b_{2})\Big)\wedge
\\
&&\rule{35mm}{0mm} \underset{x+(a_{1}+b_{1})+z=(a_{2}+b_{2})+z}{\sup} \Big(%
\delta(a_{1})\wedge\delta(a_{2})\wedge\mu(b_{1}) \wedge\mu(b_{2})\Big) \\
&&=(\lambda+_h\mu)(x)\wedge(\delta+_h\mu)(x)= \Big((\lambda+_h\mu)\wedge(%
\delta+_h\mu)\Big)(x) \\
&&=\Big((\lambda+_h\mu)\odot_{h}(\delta+_h\mu)\Big)(x) .
\end{eqnarray*}
So, $\mathcal{FL}_{R}$ is a distributive lattice.

The converse statement is a consequence of Proposition \ref{P4.3}.
\end{proof}


\section{Prime ideals}

An $h$-ideal $P$ of $R$ is called \textit{prime} if $P\ne R$ and for any $h$%
-ideals $A$, $B$ of $R$ from $AB\subseteq P$ it follows $A\subseteq P$ or $%
B\subseteq P$, and \textit{irreducible} if $P\ne R$ and $A\cap B=P$ implies $%
A=P$ or $B=P$. By analogy a non-constant fuzzy $h$-ideal $\delta$ of $R$ is
called \textit{prime} (in the first sense) if for any fuzzy $h$-ideals $%
\lambda$, $\mu$ of $R$ from $\lambda\odot_{h}\mu\leq\delta$ it follows $%
\lambda\leq\delta$ or $\mu\leq\delta$, and \textit{irreducible} if $%
\lambda\wedge\mu=\delta$ implies $\lambda=\delta$ or $\mu =\delta$.

\begin{theorem}
\label{T5.1} A left $($right$)$ $h$-ideal $P$ of $R$ is prime if and only if
for all $a,b\in R$ from $aRb\subseteq P$ it follows $a\in P$ or $b\in P$.
\end{theorem}

\begin{proof}
Assume that $P$ is a prime left $h$-ideal of $R$ and $aRb\subseteq P$ for
some $a,b\in R$. Obviously, $A=\overline{Ra}$ and $B=\overline{Rb}$ are left
$h$-ideals of $R$. So, $AB\subseteq \overline{AB}=\overline{\overline{Ra}%
\overline{Rb}}=\overline{RaRb}\subseteq \overline{RP}\subseteq P$, and
consequently $A\subseteq P$ or $B\subseteq P$. Let $\langle x\rangle $ be a
left $h$-ideal generated by $x\in R$. If $A\subseteq P$, then $\langle
a\rangle \subseteq \overline{Ra}=A\subseteq P$, whence $a\in P$. If $%
B\subseteq P$, then $\langle b\rangle \subseteq \overline{Rb}=B\subseteq P$,
whence $b\in P$.

The converse is obvious.
\end{proof}

\begin{corollary}
\label{C5.2} An $h$-ideal $P$ of $R$ is prime if and only if for all $a,b\in
R$ from $aRb\subseteq P$ it follows $a\in P$ or $b\in P$.
\end{corollary}

\begin{corollary}
\label{C5.3} An $h$-ideal $P$ of a commutative hemiring $R$ with identity is
prime if and only if for all $a,b\in R$ from $ab\in P$ it follows $%
a\in P$ or $b\in P$.
\end{corollary}

The result expressed by Corollary \ref{C5.2} suggests the following
definition of prime fuzzy $h$-ideals.

\begin{definition}
\label{D5.4} A non-constant fuzzy $h$-ideal $\delta$ of $R$ is called
\textit{prime} (in the second sense) if for all $t\in [0,1]$ and $a,b\in R$
the following condition is satisfied:

\smallskip

if $\delta(axb)\geq t$ for every $x\in R$ then $\delta(a)\geq t$ or $%
\delta(b)\geq t$.
\end{definition}

In other words, a non-constant fuzzy $h$-ideal $\delta$ is prime if from the
fact that $axb\in U(\delta;t)$ for every $x\in R$ it follows $a\in
U(\delta;t)$ or $b\in U(\delta;t)$. It is clear that any fuzzy $h$-ideal
prime in the first sense is prime in the second sense. The converse is not
true.

\begin{example}
\label{Ex5.5} In an ordinary hemiring of natural numbers the set of even
numbers forms an $h$-ideal. A fuzzy set
\begin{equation*}
\delta(n)=\left\{%
\begin{array}{ccl}
1 & \mathrm{if } & n=0, \\
0.8 & \mathrm{if } & n=2k\ne 0, \\
0.4 & \mathrm{if } & n=2k+1%
\end{array}%
\right.
\end{equation*}
is a fuzzy $h$-ideal of this hemiring. It is prime in the second sense but
it is not prime in the first sense.
\end{example}

\begin{theorem}
\label{T5.5} A non-constant fuzzy $h$-ideal $\delta$ of $R$ is prime in the
second sense if and only if each its proper level set $U(\delta;t)$ is a
prime $h$-ideal of $R$.
\end{theorem}

\begin{proof}
Let a fuzzy $h$-ideal $\delta$ of $R$ be prime in the second sense and let $%
U(\delta;t)$ be its arbitrary proper level set, i.e., $\emptyset\ne
U(\delta;t)\ne R$. If $aRb\subseteq U(\delta;t)$, then $\delta(axb)\geq t$
for every $x\in R$. Hence $\delta(a)\geq t$ or $\delta(b)\geq t$, i.e., $%
a\in U(\delta;t)$ or $b\in U(\delta;t)$, which, by Corollary \ref{C5.2},
means that $U(\delta;t)$ is a prime $h$-ideal of $R$.

To prove the converse consider a non-constant fuzzy $h$-ideal $\delta$ of $R$. If it is not prime then there exists $a,b\in R$ such that $\delta(axb)\geq
t$ for all $x\in R$, but $\delta(a)<t$ and $\delta(b)<t$. Thus, $%
aRb\subseteq U(\delta;t)$, but $a\not\in U(\delta;t)$ and $b\not\in
U(\delta;t)$. Therefore $U(\delta;t)$ is not prime. Obtained contradiction
proves that $\delta$ should be prime.
\end{proof}

\begin{corollary}
\label{C5.7} A fuzzy set $\lambda_A$ defined in Proposition $\ref{P2.8}$ is
a prime fuzzy $h$-ideal of $R$ if and only if $A$ is a prime $h$-ideal of $R$%
.
\end{corollary}

In view of the Transfer Principle (Lemma \ref{trans}) the second definition
of prime fuzzy $h$-ideals is better. Therefore fuzzy $h$-ideals which are
prime in the first sense will be called \textit{$h$-prime}.

\begin{proposition}
\label{P5.8} A non-constant fuzzy $h$-ideal $\delta$ of a commutative
hemiring $R$ with identity is prime if and only if $\,\delta(ab)=\delta(a)%
\vee\delta(b)$ for all $a,b\in R$.
\end{proposition}

\begin{proof}
Let $\delta$ be a non-constant fuzzy $h$-ideal of a commutative hemiring $R$
with identity. If $\delta(ab)=t$, then, for every $x\in R$, we have $%
\delta(axb)= \delta(xab)\geq\delta(x)\vee\delta(ab)\geq t$. Thus $%
\delta(axb)\geq t$ for every $x\in R$, which implies $\delta(a)\geq t$ or $%
\delta(b)\geq t$. If $\delta(a)\geq t$, then $t=\delta(ab)\geq\delta(a)\geq
t $, whence $\delta(ab)=\delta(a)$. If $\delta(b)\geq t$, then, as in the
previous case, $\delta(ab)=\delta(b)$. So, $\delta(ab)=\delta(a)\vee%
\delta(b) $.

Conversely, assume that $\delta(ab)=\delta(a)\vee\delta(b)$ for all $a,b\in
R $. If $\delta(axb)\geq t$ for every $x\in R$, then, replacing in this
inequality $x$ by the identity of $R$, we obtain $\delta(ab)\geq t$. Thus $%
\delta(a)\vee\delta(b)\geq t$, i.e., $\delta(a)\geq t$ or $\delta(b)\geq t$,
which means that a fuzzy $h$-ideal $\delta$ is prime.
\end{proof}

\begin{theorem}
\label{T5.9} Every proper $h$-ideal is contained in some proper irreducible $%
h$-ideal.
\end{theorem}

\begin{proof}
Let $P$ be a proper $h$-ideal of $R$ and let $\{P_{\alpha
}\,|\,\alpha\in\Lambda\}$ be a family of all proper $h$-ideals of $R$
containing $P$. By Zorn's Lemma, for any fixed $a\notin P,$ the family of $h$%
-ideals $P_{\alpha}$ such that $P\subseteq P_{\alpha}$ and $a\notin
P_{\alpha}$ contains a maximal element $M$. This maximal element is an
irreducible $h$-ideal. Indeed, let $M=P_{\beta}\cap P_{\delta}$ for some $h$%
-ideals of $R$. If $M$ is a proper subset of $P_{\beta}$ and $P_{\delta}$,
then, according to the maximality of $M$, we have $a\in P_{\beta}$ and $a\in
P_{\delta}$. Hence $a\in P_{\beta}\cap P_{\delta}=M$, which is impossible.
Thus, either $M=P_{\beta}$ or $M=P_{\delta}$.
\end{proof}

\begin{theorem}
\label{T5.10} If all $h$-ideals of $R$ are $h$-idempotent, then an $h$-ideal
$P$ of $R$ is irreducible if and only if it is prime.
\end{theorem}

\begin{proof}
Assume that all $h$-ideals of $R$ are $h$-idempotent. Let $P$ be a fixed
irreducible $h$-ideal. If $AB\subseteq P$ for some $h$-ideals $A$, $B$, then
$A\cap B=\overline{AB}\subseteq \overline{P}=P$, by Proposition \ref{P4.1}.
Thus $\overline{(A\cap B)+P}=P$. Since $\mathcal{L}_{R}$ is a distributive
lattice, $P=\overline{(A\cap B)+P}=\overline{(A+P)}\cap\overline{(B+P)}$. So
either $\overline{A+P}=P$ or $\overline{B+P}=P,$ that is, either $A\subseteq
P$ or $B\subseteq P.$

Conversely, if an $h$-ideal $P$ is prime and $A\cap B=P$ for some $A,B\in%
\mathcal{L}_{R}$, then $AB\subseteq\overline{AB}=A\cap B=P.$ Thus $%
A\subseteq P$ or $B\subseteq P$. But $P\subseteq A$ and $P\subseteq B$.
Hence $A=P$ or $B=P.$
\end{proof}

\begin{corollary}
\label{C5.11} In hemirings in which all $h$-ideals are $h$-idempotent each
proper $h$-ideal is contained in some proper prime $h$-ideal.
\end{corollary}

\begin{theorem}
\label{T5.12} In hemirings in which all fuzzy $h$-ideals are idempotent a
fuzzy $h$-ideal is irreducible if and only if it is $h$-prime.
\end{theorem}

\begin{proof}
Let all fuzzy $h$-ideals of $R$ will be idempotent and let $\delta$ be an
arbitrary irreducible fuzzy $h$-ideal of $R$. We prove that it is prime. If $%
\lambda\odot_{h}\mu\leq\delta$ for some fuzzy $h$-ideals, then also $%
\lambda\wedge\mu\leq\delta$. Since the set $\mathcal{FL}_{R}$ of all fuzzy $%
h $-ideals of $R$ is a distributive lattice (Theorem \ref{f-lattice}) we
have $\,\delta=(\lambda\wedge\mu)+_h\delta=(\lambda+_h\delta)\wedge(\mu+_h%
\delta)$. Thus $\lambda+_h\delta=\delta$ or $\mu+_h\delta=\delta$. But $\leq$
is a lattice order, so $\lambda\leq\delta$ or $\mu\leq\delta$. This proves
that a fuzzy $h$-ideal $\delta$ is $h$-prime.

Conversely, if $\delta$ is an $h$-prime fuzzy $h$-ideal of $R$ and $%
\lambda\wedge\mu=\delta$ for some $\lambda,\mu\in\mathcal{FL}_{R}$, then $%
\lambda\odot_{h}\mu=\delta$, which implies $\lambda\leq\delta$ or $%
\mu\leq\delta$. Since $\leq$ is a lattice order and $\delta=\lambda\wedge\mu$
we have also $\delta\leq\lambda$ and $\delta\leq\mu$. Thus $\lambda=\delta$
or $\mu=\delta$. So, $\delta$ is irreducible.
\end{proof}

\begin{theorem}
\label{T5.13} The following assertions for a hemiring $R$ are equivalent:

$(1)$ \ Each $h$-ideal of $R$ is $h$-idempotent.

$(2)$ \ Each proper $h$-ideal $P$ of $R$ is the intersection of all prime $h$%
-ideals containing $P$.
\end{theorem}

\begin{proof}
Let $P$ be a proper $h$-ideal of $R$ and let $\{P_{\alpha
}\,|\,\alpha\in\Lambda\}$ be the family of all prime $h$-ideals of $R$
containing $P$. Clearly $P\subseteq\cap_{\alpha\in\Lambda}P_{\alpha}$. By
Zorn's Lemma, for any fixed $a\notin P,$ the family of $h$-ideals $%
P_{\alpha} $ such that $P\subseteq P_{\alpha}$ and $a\notin P_{\alpha}$
contains a maximal element $M_a$. We will show that this maximal element is
an irreducible $h$-ideal. Let $M_a=K\cap L$. If $M_a$ is a proper subset of $%
K$ and $L$, then, according to the maximality of $M_a$, we have $a\in K$ and
$a\in L$. Hence $a\in K\cap L=M_a$, which is impossible. Thus, either $M_a=K$
or $M_a=L$. By Theorem \ref{T5.10}, $M_a$ is a prime $h$-ideal. So there
exists a prime $h$-ideal $M_a$ such that $a\notin M_a$ and $P\subseteq M_a$.
Hence $\cap P_{\alpha }\subseteq P.$ Thus $P=\cap P_{\alpha }$.

Assume that each $h$-ideal of $R$ is the intersection of all prime $h$%
-ideals of $R$ which contain it. Let $A$ be an $h$-ideal of $R.$ If $%
\overline{A^{2}}=R$, then, by Lemma \ref{L2.3}, we have $A=R$, which means
such $h$-ideal is $h$-idempotent. If $\overline{A^{2}}\neq R,$ then $%
\overline{A^{2}}$ is a proper $h$-ideal of $R$ and so it is the intersection
of all prime $h$-ideals of $R$ containing $A$. Let $\overline{A^{2}}=\cap
P_{\alpha }$. Then $A^{2}\subseteq P_{\alpha }$ for each $\alpha$. Since $%
P_{\alpha }$ is prime, we have $A\subseteq P_{\alpha }$. Thus $A\subseteq
\cap P_{\alpha }=\overline{A^{2}}$. But $\overline{A^{2}}\subseteq A$ for
every $h$-ideal. Hence $A=\overline{A^{2}}$.
\end{proof}

\begin{lemma}
\label{L5.14} Let $R$ be a hemiring in which each fuzzy $h$-ideal is
idempotent. If $\lambda $ is a fuzzy $h$-ideal of $R$ with $%
\lambda(a)=\alpha ,$ where $a$ is any element of $R$ and $\alpha \in \left[
0,1\right] $, then there exists an irreducible and $h$-prime fuzzy $h$-ideal
$\delta $ of $R$ such that $\lambda \leq \delta $ and $\delta(a)=\alpha$.
\end{lemma}

\begin{proof}
Let $\lambda$ be an arbitrary fuzzy $h$-ideal of $R$ and let $a\in R$ be
fixed. Consider the following collection of fuzzy $h$-ideals of $R$
\begin{equation*}
\mathcal{B}=\{\mu\,|\,\mu(a)=\lambda(a), \ \lambda\leq\mu\} .
\end{equation*}
$\mathcal{B}$ \ is non-empty since $\lambda\in\mathcal{B}$. Let $\mathcal{F}$
be a totally ordered subset of $\mathcal{B}$ containing $\lambda$, say $%
\mathcal{F}=\{\lambda _{i}\,|\,i\in I\}$. Obviously $\lambda_i\vee\lambda_j%
\in \mathcal{F}$ for any $\lambda_i,\lambda_j\in\mathcal{F}$. So, for
example, $\big(\lambda_{i}(x)\vee\lambda_{j}(x)\big)\wedge \big(%
\lambda_{i}(y)\vee\lambda_{j}(y)\big)\leq
\lambda_{i}(x+y)\vee\lambda_{j}(x+y)$ for any $\lambda_i,\lambda_j\in%
\mathcal{F}$ and $x,y\in R$.

We claim that $\bigvee\limits_{i\in I}\lambda _{i}$ is a fuzzy $h$-ideal of $%
R$.

For any $x,y\in R,$ we have
\begin{eqnarray*}
\big(\bigvee\limits_{i\in I}\lambda_{i}\big)(x)\wedge \big(%
\bigvee\limits_{i\in I}\lambda_{i}\big)(y) &=&\big(\bigvee\limits_{i\in
I}\lambda_{i}(x)\big)\wedge \big(\bigvee\limits_{j\in I}\lambda_{j}(y)\big)
\\
&=&\bigvee\limits_{i,j\in I}\big(\lambda_{i}(x)\wedge\lambda _{j}(y)\big) \\
&\leq &\bigvee\limits_{i,j\in I}\Big( \big(\lambda_{i}(x)\vee\lambda_{j}(x)%
\big)\wedge \big(\lambda_{i}(y)\vee\lambda_{j}(y)\big)\Big) \\
&\leq &\bigvee\limits_{i,j\in I }\big(\lambda_{i}(x+y)\vee\lambda_{j}(x+y)%
\big) \\
&\leq &\bigvee\limits_{i\in I}\lambda_{i}(x+y)=\big(\bigvee\limits_{i\in
I}\lambda_{i}\big)(x+y)\, .
\end{eqnarray*}
Similarly
\begin{equation*}
\big(\bigvee\limits_{i\in I}\lambda_{i}\big)(x) =\bigvee\limits_{i\in
I}\lambda_{i}(x)\leq \bigvee\limits_{i\in I}\lambda_{i}(xr) =\big(%
\bigvee\limits_{i\in I}\lambda _{i}\big)(xr)
\end{equation*}
and
\begin{equation*}
\big(\bigvee\limits_{i\in I}\lambda_{i}\big)(x)\leq \big(\bigvee\limits_{i%
\in I}\lambda_{i}\big)(rx)
\end{equation*}
for all $x,r\in R.$ Thus $\bigvee\limits_{i\in I}$ is a fuzzy ideal.

Now, let $x+a+z=b+z$, where $a,b,z\in R.$ Then
\begin{eqnarray*}
\big(\bigvee\limits_{i\in I}\lambda _{i}\big)(a)\wedge \big(%
\bigvee\limits_{i\in I}\lambda _{i}\big)(b) &=&\big(\bigvee\limits_{i\in
I}\lambda _{i}(a)\big)\wedge \big(\bigvee\limits_{j\in I}(\lambda _{j}(b)%
\big) \\
&=&\bigvee\limits_{i,j\in I}\big(\lambda _{i}(a)\wedge \lambda _{j}(b)\big)
\\
&\leq &\bigvee\limits_{i,j\in I}\Big(\big(\lambda _{i}(a)\vee \lambda _{j}(a)%
\big)\wedge \big(\lambda _{i}(b)\vee \lambda _{j}(b)\big)\Big) \\
&\leq &\bigvee\limits_{i,j}\big(\lambda _{i}(x)\vee \lambda _{j}(x)\big)\leq
\bigvee\limits_{i\in i}\lambda _{i}(x)=\big(\bigvee\limits_{i\in I}\lambda
_{i}\big)(x)\,.
\end{eqnarray*}%
This means that $\bigvee\limits_{i\in I}\lambda _{i}$ is a fuzzy $h$-ideal
of $R$. Clearly $\lambda \leq \bigvee\limits_{i\in I}\lambda _{i}$ and $%
(\bigvee\limits_{i\in I}\lambda _{i})(a)=\lambda (a)=\alpha $. Thus $%
\bigvee\limits_{i\in I}\lambda _{i}$ is the least upper bound of $\mathcal{F}
$. Hence by Zorn's lemma there exists a fuzzy $h$-ideal $\delta $ of $R$
which is maximal with respect to the property that $\lambda \leq \delta $
and $\delta (a)=\alpha $.

We will show that $\delta $ is an irreducible fuzzy $h$-ideal of $R$. Let $%
\delta=\delta_{1}\wedge\delta_{2},$ where $\delta_{1}$, $\delta_{2}$ are
fuzzy $h$-ideals of $R$. Then $\delta\leq\delta_{1}$ and $%
\delta\leq\delta_{2}$ since $\mathcal{FL}_R$ is a lattice. We claim that
either $\delta=\delta_{1}$ or $\delta=\delta_{2}$. Suppose $%
\delta\neq\delta_{1}$ and $\delta\neq\delta_{2}$. Since $\delta $ is maximal
with respect to the property that $\delta(a)=\alpha$ and since $%
\delta\lneqq\delta_{1}$ and $\delta\lneqq\delta_{2}$, so $%
\delta_{1}(a)\neq\alpha$ and $\delta_{2}(a)\neq\alpha$. Hence $%
\alpha=\delta(a)=(\delta_{1}\wedge\delta_{2})(a)=
\delta_{1}(a)\wedge\delta_{2}(a)\neq\alpha$, which is impossible. Hence $%
\delta=\delta_{1}$ or $\delta=\delta_{2}$. Thus $\delta$ is an irreducible
fuzzy $h$-ideal of $R$. By Theorem \ref{T5.12}, it is also prime.
\end{proof}

\begin{theorem}
\label{T5.15} Each fuzzy $h$-ideal of $R$ is idempotent if and
only if each fuzzy $h$-ideal of $R$ is the intersection of those
$h$-prime fuzzy $h$-ideals of $R$ which contain it.
\end{theorem}

\begin{proof}
Suppose each fuzzy $h$-ideal of $R$ is idempotent. Let $\lambda $ be a fuzzy
$h$-ideal of $R$ and let $\{\lambda _{\alpha }\,|\,\alpha \in \Lambda \}$ be
the family of all $h$-prime fuzzy $h$-ideals of $R$ which contain $\lambda $%
. Obviously $\lambda \leq \bigwedge\limits_{\alpha \in \Lambda }\lambda
_{\alpha }$. We now show that $\bigwedge\limits_{\alpha \in \Lambda }\lambda
_{\alpha }\leq \lambda $. Let $a$ be an arbitrary element of $R$. Then,
according to Lemma \ref{L5.14}, there exists an irreducible and $h$-prime
fuzzy $h$-ideal $\delta $ such that $\lambda \leq \delta $ and $\lambda
(a)=\delta (a)$. Hence $\delta \in \{\lambda _{\alpha }\,|\,\alpha \in
\Lambda \}$ and $\bigwedge\limits_{\alpha \in \Lambda }\lambda _{\alpha
}\leq \delta $. So, $\bigwedge\limits_{\alpha \in \Lambda }\lambda _{\alpha
}(a)\leq \delta (a)=\lambda (a)$. Thus $\bigwedge\limits_{\alpha \in \Lambda
}\lambda _{\alpha }\leq \lambda $. Therefore $\bigwedge\limits_{\alpha \in
\Lambda }\lambda _{\alpha }=\lambda $.

Conversely, assume that each fuzzy $h$-ideal of $R$ is the
intersection of those $h$-prime fuzzy $h$-ideals of $R$ which
contain it. Let $\lambda $ be a fuzzy $h$-ideal of $R$ then
$\lambda \odot \lambda $ is also fuzzy $h$-ideal of $R$, so
$\lambda \odot \lambda =\bigwedge\limits_{\alpha \in \Lambda
}\lambda_{\alpha }$ where $\lambda _{\alpha }$ are $h$-prime fuzzy
$h$-ideals of $R$. Thus each $\lambda _{\alpha }$ contains
$\lambda \odot \lambda $, and hence $\lambda $. So $\lambda
\subseteq $ $\bigwedge\limits_{\alpha \in \Lambda }\lambda
_{\alpha }=\lambda \odot \lambda $, but $\lambda \odot \lambda
\subseteq \lambda $ always. Hence $\lambda=\lambda\odot\lambda .$
\end{proof}


\section{Semiprime ideals}


\begin{definition}
An $h$-ideal $A$ of $R$ is called \textit{semiprime} if $A\ne R$ and for any
$h$-ideal $B$ of $R$, $B^{2}\subseteq A$ implies $B\subseteq A$. Similarly,
a non-constant fuzzy $h$-ideal $\lambda $ of $R$ is called \textit{semiprime}
if for any fuzzy $h$-ideal $\delta $ of $R,$ \ $\delta\odot_{h}\delta\leq%
\lambda$ implies $\delta\leq\lambda$.
\end{definition}

Obviously, each semiprime $h$-ideal is prime. Each semiprime fuzzy $h$-ideal
is $h$-prime. The converse is not true (see Example \ref{Ex6.7}).

Using the same method as in the proof of Theorem \ref{T5.1} we can prove

\begin{theorem}
\label{T6.2} A $($left, right$)$ $h$-ideal $P$ of $R$ is semiprime if and
only if for every $a\in R$ from $aRa\subseteq P$ it follows $a\in P.$
\end{theorem}

\begin{corollary}
\label{C6.3} An $h$-ideal $P$ of a commutative hemiring $R$ with identity is
semiprime if and only if for all $a\in R$ from $a^2\in P$ it follows $%
a\in P$.
\end{corollary}

\begin{theorem}
\label{T6.4} The following assertions for a hemiring $R$ are equivalent:

$(1)$ Each $h$-ideal of $R$ is $h$-idempotent.

$(2)$ Each $h$-ideal of $R$ is semiprime.
\end{theorem}

\begin{proof}
Suppose that each $h$-ideal of $R$ is idempotent. Let $A$, $B$ be $h$-ideals
of $R$ such that $B^{2}\subseteq A.$ Thus $\overline{B^{2}}\subseteq
\overline{A}=A. $ By hypothesis $B=\overline{B^{2}},$ so $B\subseteq A.$
Hence $A$ is semiprime.

Conversely, assume that each $h$-ideal of $R$ is semiprime. Let $A$ be an $h$%
-ideal of $R.$ Then $\overline{A^{2}}$ is also an $h$-ideal of $R.$ Also $%
A^{2}\subseteq\overline{A^{2}}.$ Hence by hypothesis $A\subseteq\overline{%
A^{2}}.$ But $\overline{A^{2}}\subseteq A$ always. Hence $A=\overline{A^{2}}$%
.
\end{proof}

\begin{theorem}
\label{T6.5} Each fuzzy $h$-ideal of $R$ is idempotent if and only if each
fuzzy $h$-ideal of $R$ is semiprime.
\end{theorem}

\begin{proof}
For any $h$-ideal of $R$ we have $\lambda\odot_{h}\lambda\leq\lambda$
(Theorem \ref{T3.3}). If each $h$-ideal of $R$ is semiprime, then $%
\lambda\odot_{h}\lambda\leq\lambda\odot_{h}\lambda$ implies $%
\lambda\leq\lambda\odot_{h}\lambda$. Hence $\lambda\odot_{h}\lambda=\lambda$.

The converse is obvious.
\end{proof}

Below we present two examples of hemirings in which all fuzzy $h$-ideals are
semiprime.

\begin{example}
\label{Ex6.6} Consider the set $R=\{0,a,1\}$ with the following two
operations:
\begin{equation*}
\begin{tabular}{c|ccc}
$+$ & $0$ & $a$ & $1$ \\ \hline
$0$ & $0$ & $a$ & $1$ \\
$a$ & $a$ & $a$ & $a$ \\
$1$ & $1$ & $a$ & $1$%
\end{tabular}
\ \ \ \ \ \ \ \ \ \ \ \ \ \ \ \ \ \ \ \
\begin{tabular}{c|ccc}
$\cdot $ & $0$ & $a$ & $1$ \\ \hline
$0$ & $0$ & $0$ & $0$ \\
$a$ & $0$ & $a$ & $a$ \\
$1$ & $0$ & $a$ & $1$%
\end{tabular}%
\end{equation*}

Then $(R,+,\cdot)$ is a commutative hemiring with identity. It has only one
proper ideal $\{0,a\}$. This ideal is not an $h$-ideal. The only $h$-ideal
of $R$ is $\{0,a,1\}$, which is clearly $h$-idempotent.

Since $0=0a=a0=01=10,$ for any fuzzy ideal $\lambda$ of this hemiring we
have $\lambda(0)\geq\lambda(a)$ and $\lambda(0)\geq\lambda(1)$ and $%
\lambda(a)=\lambda(1a)\geq\lambda(1)$. Thus $\lambda(0)\geq\lambda(a)\geq%
\lambda(1)$. If $\lambda$ is a fuzzy $h$-ideal, then $1+0+1=0+1$ implies $%
\lambda(1)\geq\lambda(0)\wedge\lambda(0)=\lambda(0)$, which proves that each
fuzzy $h$-ideal of this hemiring is a constant function. So, $%
\lambda\odot_{h}\lambda=\lambda$ for each fuzzy $h$-ideal $\lambda$ of $R$.
This, by Theorem \ref{T6.5}, means that each fuzzy $h$-ideal of $R$ is
semiprime.
\end{example}

\begin{example}
\label{Ex6.7} Now, consider the hemiring $R=\{0,a,b,c\}$ defined by the
following tables:
\begin{equation*}
\begin{tabular}{c|cccc}
$+$ & $0$ & $a$ & $b$ & $c$ \\ \hline
$0$ & $0$ & $a$ & $b$ & $c$ \\
$a$ & $a$ & $b$ & $c$ & $a$ \\
$b$ & $b$ & $c$ & $a$ & $b$ \\
$c$ & $c$ & $a$ & $b$ & $c$%
\end{tabular}
\ \ \ \ \ \ \ \ \ \
\begin{tabular}{c|cccc}
$\cdot $ & $0$ & $a$ & $b$ & $c$ \\ \hline
$0$ & $0$ & $0$ & $0$ & $0$ \\
$a$ & $0$ & $a$ & $b$ & $c$ \\
$b$ & $0$ & $b$ & $b$ & $c$ \\
$c$ & $0$ & $c$ & $b$ & $c$%
\end{tabular}%
\end{equation*}

This hemiring has only one $h$-ideal $A=R$. Obviously this $h$-ideal is $h$%
-idempotent.

For any fuzzy ideal $\lambda$ of $R$ and any $x\in R$ we have $%
\lambda(0)\geq\lambda(x)\geq\lambda(a)$. Indeed, $\lambda(0)=\lambda(0x)\geq%
\lambda(x)=\lambda(xa)\geq\lambda(a)$. This together with $%
\lambda(a)=\lambda(b+b)\geq\lambda(b)\wedge\lambda(b)=\lambda(b)$ implies $%
\lambda(a)=\lambda(b)$. Consequently, $\lambda(c)=\lambda(a+b)\geq\lambda(a)%
\wedge\lambda(b)=\lambda(b)$. Therefore $\lambda(0)\geq\lambda(c)\geq%
\lambda(b)=\lambda(a)$. Moreover, if $\lambda$ is a fuzzy $h$-ideal, then $%
c+0+a=0+a$, which implies $\lambda(c)\geq\lambda(0)\wedge\lambda(0)=%
\lambda(0)$. Thus $\lambda(0)=\lambda(c)\geq\lambda(b)=\lambda(a)$ for every
fuzzy $h$-ideal of this hemiring.

Now we prove that each fuzzy $h$-ideal of $R$ is idempotent. Since $%
\lambda\odot_{h}\lambda\leq\lambda$ always, so we have to show that $%
\lambda\odot_{h}\lambda\geq\lambda$. Obviously, for every $x\in R$ we have
\begin{equation*}
\arraycolsep=.5mm%
\begin{array}{ll}
(\lambda\odot_{h}\lambda)(x) & =\!\underset{ x+\sum
\limits_{i=1}^{m}a_{i}b_{i}+z=\sum\limits_{j=1}^{n}a_{j}^{\prime
}b_{j}^{\prime }+z}{\sup} \!\!\Big( \bigwedge\limits_{i=1}^m\big(
\lambda(a_{i})\wedge\lambda(b_{i})\big)\wedge
\bigwedge\limits_{j=1}^n\big(
\lambda(a_{j}^{\prime })\wedge\lambda(b_{j}^{\prime })\big) \Big) \\
& \geq\underset{x+cd+z=c^{\prime }d^{\prime }+z}{\sup}\big(
\lambda(c)\wedge\lambda(d)\wedge \lambda(c^{\prime
})\wedge\lambda(d^{\prime })\big)=\lambda(c)\wedge\lambda(d)\wedge
\lambda(c^{\prime })\wedge\lambda(d^{\prime }).
\end{array}
\end{equation*}

So, $x+cd+z=c^{\prime }d^{\prime }+z$ implies $(\lambda \odot _{h}\lambda
)(x)\geq \lambda (c)\wedge \lambda (d)\wedge \lambda (c^{\prime })\wedge
\lambda (d^{\prime })$. Hence $0+00+z=00+z\,$ implies $(\lambda \odot
_{h}\lambda )(0)\geq \lambda (0)$. Similarly $a+bb+z=bc+z$ implies $(\lambda
\odot _{h}\lambda )(a)\geq \lambda (b)\wedge \lambda (c)=\lambda (b)=\lambda
(a)$, $b+aa+z=bc+z\,$ implies $(\lambda \odot _{h}\lambda )(b)\geq \lambda
(a)\wedge \lambda (b)\wedge \lambda (c)=\lambda (b)$. Analogously, from $%
c+00+z=cc+z\,$ it follows $(\lambda \circ _{h}\lambda )(c)\geq \lambda
(0)\wedge \lambda (c)=\lambda (c)$. This proves that $(\lambda \odot
_{h}\lambda )(x)\geq \lambda (x)$ for every $x\in R$. Therefore $\lambda
\odot _{h}\lambda =\lambda $ for every fuzzy $h$-ideal of $R$, which, by
Theorem \ref{T6.5}, means that each fuzzy $h$-ideal of $R$ is semiprime.

Consider the following three fuzzy sets:
\begin{equation*}
\begin{array}{cc}
\lambda (0)=\lambda (c)=0.8, & \lambda (a)=\lambda (b)=0.4, \\
\mu (0)=\mu (c)=0.6, & \mu (a)=\mu (b)=0.5, \\
\delta (0)=\delta (c)=0.7, & \delta (a)=\delta (b)=0.45.%
\end{array}%
\end{equation*}%
These three fuzzy sets are idempotent fuzzy $h$-ideals. Since all fuzzy $h$%
-ideal of this hemiring are idempotent, by Proposition \ref{P4.3}, we have $%
\lambda \odot _{h}\mu =\lambda \wedge \mu $. Thus $(\lambda \odot
_{h}\mu )(0)=(\lambda \odot _{h}\mu )(c)=0.6$ and $(\lambda \odot
_{h}\mu )(a)=(\lambda \odot _{h}\mu )(b)=0.4.$ So, $\lambda \odot
_{h}\mu \leq \delta $ but neither $\lambda \leq \delta $ nor $\mu
\leq \delta $, that is $\delta $ is not an $h$-prime fuzzy
$h$-ideal.
\end{example}

Theorem \ref{T6.2} suggests the following definition of semiprime
fuzzy $h$-ideals.

\begin{definition}
\label{D6.8} A non-constant fuzzy $h$-ideal $\delta$ of $R$ is called
\textit{semiprime} (in the second sense) if for all $t\in [0,1]$ and $a,b\in
R$ the following condition is satisfied:

\smallskip

if $\delta(axb)\geq t$ for every $x\in R$ then $\delta(a)\geq t$ or $%
\delta(b)\geq t$.
\end{definition}

In other words, a non-constant fuzzy $h$-ideal $\delta$ is semiprime if from
the fact that $axb\in U(\delta;t)$ for every $x\in R$ it follows $a\in
U(\delta;t)$ or $b\in U(\delta;t)$. It is clear that any fuzzy $h$-ideal
semiprime in the first sense is semiprime in the second sense. The converse
is not true (see Example \ref{Ex5.5}).

\begin{theorem}
\label{T6.9} A non-constant fuzzy $h$-ideal $\delta$ of $R$ is semiprime in
the second sense if and only if each its proper level set $U(\delta;t)$ is a
semiprime $h$-ideal of $R$.
\end{theorem}

\begin{proof}
The proof is analogous to the proof of Theorem \ref{T5.5}.
\end{proof}

\begin{corollary}
\label{C6.10} A fuzzy set $\lambda_A$ defined in Proposition $\ref{P2.8}$ is
a semiprime fuzzy $h$-ideal of $R$ if and only if $A$ is a semiprime $h$%
-ideal of $R$.
\end{corollary}

In view of the Transfer Principle (Lemma \ref{trans}) the second definition
of semiprime fuzzy $h$-ideals is better. Therefore fuzzy $h$-ideals which
are prime in the first sense should be called \textit{$h$-semiprime}.

\begin{proposition}
\label{P6.11} A non-constant fuzzy $h$-ideal $\delta$ of a commutative
hemiring $R$ with identity is semiprime if and only if $\,\delta(a^2)=%
\delta(a)$ for every $a\in R$.
\end{proposition}

\begin{proof}
The proof is similar to the proof of Proposition \ref{P5.8}.
\end{proof}

\section{Conclusion} In the study of fuzzy algebraic system,
the fuzzy ideals with special properties always play an important
role. In this paper we study those hemirings for which each fuzzy
$h$-ideal is idempotent. We characterize these hemirings in terms
of prime and semiprime fuzzy $h$-ideals. In the future we wanted
to study those hemirings for which each fuzzy one sided $h$-ideal
is idempotent and also those hemirings for which each fuzzy
$h$-bi-ideal is idempotent. We also want to establish a fuzzy
spectrum of hemirings.

We hope that the research along this direction can be continued, and our results presented in this paper have already constituted a platform for further discussion concerning the future development of hemirings and their applications to study fundamental concepts of the automata theory such as nondeterminism, for example.

\end{document}